\documentclass[15pt]{amsart}
\usepackage[all,cmtip]{xy}
\usepackage{amsthm}
\usepackage{amssymb}
\usepackage{xypic}
\usepackage{graphicx}
\usepackage{amscd}
\usepackage{setspace}
\usepackage{mathrsfs}
\usepackage[colorlinks=false, linkcolor=red]{hyperref}
\usepackage{enumerate}
\usepackage{textcomp}
\theoremstyle{plain}
\newtheorem{theorem}{Theorem}[section]
\newtheorem{definition}[theorem]{Definition}

\newtheorem{question}[theorem]{Question}
\newtheorem{proposition}[theorem]{Proposition}

\newtheorem{lemma}[theorem]{Lemma}
\newtheorem{remark}[theorem]{Remark}
\numberwithin{equation}{theorem}

\newcommand{\PP}{\mathbb{P}}
\newcommand{\EE}{\mathcal{E}}

\newcommand{\OO}{\mathcal{O}}

\newcommand{\LL}{\mathcal{L}}
\newcommand{\Ql}{\mathbb{Q}_l}
\newcommand{\et}{\mathrm{H}_{\text{\'et}}}
\newcommand{\XX}{\mathcal{X}}
\newcommand{\mA}{\mathcal{A}^{\vee}}
\newcommand{\cris}{\mathrm{H}_{\mathrm{cris}}}
\newcommand{\NS}{\mathrm{NS}}
\newcommand{\X}{\overline{X}}
\newcommand{\Z}{\mathbb{Z}}
\newcommand{\CC}{\mathbb{C}}

\DeclareMathOperator {\Spec}{\textnormal{Spec}}
\DeclareMathOperator {\Pic}{Pic}
\DeclareMathOperator {\Aut}{Aut}

\DeclareMathOperator {\Def}{Def}

\DeclareMathOperator {\Dim}{dim}
\DeclareMathOperator {\Hom}{Hom}
\DeclareMathOperator {\Sym}{Sym}

\DeclareMathOperator{\Char}{char}

\DeclareMathOperator{\Id}{Id}

\DeclareMathOperator{\Ho}{H}
\DeclareMathOperator{\DR}{DR}
\DeclareMathOperator{\Cov}{Cov}
\DeclareMathOperator{\Sets}{\textbf{Sets}}
\DeclareMathOperator{\Hilb}{Hilb}

\numberwithin{equation}{theorem}
\title []{Remarks on Automorphism and Cohomology of Finite Cyclic Coverings of Projective Spaces}

\author {renjie lyu}
\address{Korteweg-de Vries Instituut, Universiteit van Amsterdam,
Science Park 105-107, 1098 XG, Amsterdam, the Netherlands}
\email{R.Lyu@uva.nl}

\author {xuanyu pan}
\address{Institute of Mathematics, AMSS, Chinese Academy of Sciences,
55 ZhongGuanCun East Road, Beijing, 100190, China}
\email{pan@amss.ac.cn}
\email{panxuanyu@mpim-bonn.mpg.de}
\date{\today}

\begin{document}
\begin{abstract}
For a smooth finite cyclic covering over a projective space of dimension great than one, we show that the group of automorphisms acts faithfully on the cohomology except for a few cases. In characteristic zero, we study the equivariant deformation theory and the automorphism group for complex cyclic coverings. The proof uses the decomposition of the sheaf of differential forms due to Esnault and Viehweg, In positive characteristics, a lifting criterion of automorphisms reduces the problem to characteristic zero. To apply this criterion, we prove the degeneration of the Hodge-de Rham spectral sequences and the infinitesimal Torelli theorem for finite cyclic coverings of projective spaces defined over an arbitrary field.
 
\end{abstract}

\maketitle
\tableofcontents

\section{Introduction}
The Torelli theorem says that: An isomorphism $\varphi:\Ho^n(X)\simeq \Ho^n(X')$ between the cohomology groups of two smooth projective varieties $X$ and $X'$ with dimension $n$, which preserves some algebraic structures (e.g. Hodge structures) is induced by an isomorphism $\psi : X\simeq X'$ between the varieties. 

It is natural to ask whether the map $\psi$ that induces $\varphi$ is (up to a sign) unique, i.e.
\begin{question}\label{question}
 Does the automorphism group $\mathrm{Aut}(X)$ act on the cohomology group $\Ho^n(X)$ faithfully for a smooth projective variety $X$?
\end{question}
Recently, this question for varieties of complete intersections is independently studied and sloved in \cite{JA2} and \cite{Pan15}. Javanpeykar and Loughran use this faithfulness result to induce the Shafarevich conjecture for complete intersections by the Lang-Vojta conjecture, see \cite{JA1}. In \cite{JA2}, this result is used to show that the stack of smooth hypersurfaces with a level $N$ structure is uniformisable by a smooth affine scheme. The second author solves this question for smooth complex cubic fourfolds in \cite{Pan} and uses it to relate the symmetry of the defining equation of a cubic fourfold to its middle Picard number.

Historically the answer to this fundamental question for curves is now known as the Torelli theorem of curves of genus at least two. For $K$3 surfaces over different base fields, the question has been settled by Burns, Rapoport, Shafarevich and Ogus, etc. We refer to \cite[Chapter 15]{K3} for a complete survey. We know very few about Question \ref{question} for other families of varieties of higher dimensions. The aim of this paper is to study this automorphism representation for finite cyclic coverings of projective spaces. We refer to Section $1$ as preliminaries of finite cyclic coverings. Our first main theorem is 
\begin{theorem}\label{M1}
Let $\CC$ be the complex numbers, and let $f: X\rightarrow \PP^n_{\CC}$ be a smooth finite cyclic covering of the projective space with $n\geq 2$. Suppose that $X$ is not a quadric hypersurface. Then any non-trivial automorphim of $X$ non-trivially act on the cohomology group $H^n(X, \CC)$.
\end{theorem}
In Section $5$, the main result is generalized to finite cyclic coverings over a field of positive characteristic :
\begin{theorem} \label{M2}
Let $K$ be an algebraically closed field of positive characteristic, and let X be a smooth finite cyclic covering of $\PP^n_{K}$ with $n\geq 2$. If $X$ is not a quadric hypersurface, then the action of the automorphism group $\Aut(X)$ on $\et^n(X,\Ql)$ ($l\neq char(K)$) is faithful.
\end{theorem} 
The main idea of the proof involves the following lifting problem. Let $\pi: \mathcal{X}\rightarrow S$ be certain family of algebraic varieties over an irreducible base $S$, and let $X$ be a special fiber of the family. Assume that $\sigma\in \Aut(X)$ is a finite order automorphism of the special fiber $X$. A general question is whether one can lift $\sigma$ to be an automorphism of the family $\pi$, i.e. an automorphism $\tilde{\sigma}$ of $\mathcal{X}$ such that $\pi\circ \tilde{\sigma}=\pi$ and $\tilde{\sigma}|_{X}=\sigma$. If the answer to this question is affirmative, by specialization argument, one can reduce the faithfulness problem (\ref{question}) of a special fiber to a very general one. 

Let $X$ be a smooth projective variety and $G$ be a finite subgroup of automorphisms. The equivariant deformation theory of the pair $(X, G)$ is reviewed in Section $2$, where we apply Bertin and M\'ezard's result \cite{BM00} to compute the first-order deformation and obstructions of the equivaraint deformation functor $\Def_X^G$. 

The automorphism subgroup $G$ we concern in Question \ref{question} is a finite subgroup of the kernel of the automorphism representation. Hence we set $X$ to be a smooth complex cyclic covering of a projective space $X$ and $G$ to be a finite group acting trivially on the Betti cohomology group $H^*(X,\CC)$. Under the hypotheses, we apply Hodge theory and Kodaira-Spencer map to study when the equivaraint deformation functor $\Def_X^G$ is unobstructed, see Theorem \ref{thmunobs}. The deformation theory and  infinitesimal Torelli theorem of complex cyclic coverings, established by J. Wehler in \cite{Weh86}, is the key to our proof. 

Section $3$ is devoted to the automorphism representation of a very general cyclic covering. In fact, Theorem \ref{finiteauto} show that any automorphism of a very general complex cyclic covering of projective space of dimension at least $3$ is a covering transformation except quadrics. Finally by the decomposition of the sheaf of differential forms on cyclic coverings (\ref{decom-diff})(cf. \cite{Kon85}, \cite{Weh86} and \cite{EV92}), it is easy to prove any non-trivial covering transformation acts non-trivially on the cohomology $H^n(X, \CC)$, see Lemma \ref{Covfaith}. However, the case of 2-dimensional cyclic coverings will be treated separately since the conclusions of the above arguments fail to validate that case.

For a cyclic covering of projective space $X$ defined over a field $K$ of positive characteristic, we consider the induced cyclic covering defined over the Witt ring $W(K)$. Then the geometric generic fiber is a complex cyclic covering of projective space. The main point here is a lifting criterion of automorphisms of varieties defined over a field of positive characteristic, addressed by the second author in \cite[Theorem 1.7]{Pan16}(cf. Theorem \ref{corolifting}). If this criterion holds for finite cyclic coverings, we might reduce Theorem \ref{M2} to Theorem \ref{M1}.  In order to apply this criterion, on the one hand, we mimic Deligne's method in \cite[Exp XI]{SGA7-II} to show that the relative Hodge-de Rham spectral sequences 
\[
E^{j,i}_1:=R^i\pi_*\Omega^j_{\mathcal{X}/S}\Longrightarrow \mathbb{R}^{i+j}\pi_*(\Omega_{\mathcal{X}/S}^\bullet)
\]
degenerates at $E_1$-page for any family of smooth cyclic coverings of projective spaces $\pi: \mathcal{X}\rightarrow S$.
On the other hand, we need to verify the infinitesimal Torelli theorem for $X/K$, namely the canonical cup-product map
\[
\lambda_p: H^1(X, \Theta_{X/K})\rightarrow \Hom(H^{n-p}(X, \Omega_{X/K}^p), H^{n+1-p}(X,\Omega_{X/K}^{p-1}))
\]
is injective for some $p$, where $\Theta_{X/K}$ is the tangent sheaf  of $X$. As mentioned above, J. Wehler proved the infinitesimal Torelli theorem for complex cyclic coverings by using Flenner's criterion, see \cite{Fle86} and \cite{Weh86}. In this article, we use an enhanced Flenner's criterion, due to D.Zhang \cite{Pan15}, to show that Wehler's proof also works for cyclic coverings over an arbitrary field. The above two discussions are settled in Section $4$.

\textbf{Acknowledgments.} The authors are very grateful to Prof.~K.~Zuo for his interests in this paper. The authors also appreciate Prof.~M.~Kerr for his support of the algebraic geometry and Hodge theory seminar in Washington University in St.Louis. The first author is very grateful to his advisor Prof.~M.~Shen for some dicussions. The authors also thank their friend Dr.~D.~Zhang for answering many questions. Some parts of this paper were written in Max Planck Institute for Mathematics. The second author is very grateful to the institute for providing the comfortable environments.
\section{Finite Cyclic Coverings}\label{sectionfcc}

In this section, we review some basic facts of the cyclic coverings of smooth projective varieties.

\begin{definition}\label{def2.1}
Let $Z$ be a smooth projective variety over an algebraically closed field $K$, and let $\LL$ be an invertible sheaf on $Z$. Assume that $k$ is an integer and $\LL^k$ has a nontrivial section $s\in H^0(Z, \LL^k)$. There is a natural $\mathcal{O}_Z$-algebra
\[\mathcal{A}:=\bigoplus\limits_{i=0}^{k-1} \mathcal{L}^{-i},\]
where the multiplication structure is given by the section $s^{\vee}: \mathcal{L}^{-k}\longrightarrow \mathcal{O}_Z.$
We define the affine morphism associated to the invertible sheaf $\LL$\begin{equation}
f: X: =\underline{\Spec}(\mathcal{A})\rightarrow Z
\end{equation} to be the k-fold cyclic covering of $Z$ branched along the zero divisor $D:=Z(s)$.  

\end{definition}

We denote by $\mathbb{V}(L) :=\underline{\Spec}(\Sym^{\bullet}\LL^{\vee})$ the total space of the invertible sheaf $\mathcal{L}$, and let $\pi_L: \mathbb{V}(L)\rightarrow Z$ be the natural projection. If $t\in \Gamma(\mathbb{V}(L), \pi^*_L\LL)$ is the tautological section, then the $k$-fold cyclic covering $X$ is defined by the equation \[t^k-\pi_L^* s\] in $\mathbb{V}(L)$.
In particular, let $\{U_\alpha\}$ be an affine open cover of $Z$ such that $\mathcal{L}|_{U_\alpha}\simeq \OO_{U_\alpha}$. Assume that $D$ is defined by the equation $\Phi_\alpha(\underline{z})=0$ on $U_\alpha$. Then $X$ is locally defined by the equation 
\begin{equation} \label{eqforvl}
\omega_{\alpha}^k-\Phi_\alpha(\underline{z})=0,
\end{equation} where $(\underline{z}, \omega_\alpha)$ are the local coordinates on $\mathbb{V}(L)|_{U_\alpha}=U_{\alpha}\times \mathbb{A}^1$. If $k$ is not divided by $\Char(K)$, then it follows from the equation (\ref{eqforvl}) that $X$ is smooth if $D$ is smooth.

Let $\EE$ be the locally free sheaf $\OO_Z\oplus \LL^{-1}$. Denote by $\hat{L}$ be the projective bundle $\PP(\EE)$ over $Z$. The cyclic covering $X$ is a divisor in $\hat{L}$ through natural inclusions
\begin{equation}\label{linebundleL}
\xymatrix{
X \ar@{^{(}->}[r] \ar[rd]_f  &\mathbb{V}(L) \ar[d]_{\pi_L} \ar@{^(->}[r] & \hat{L} \ar[dl]^{\pi}\\
&Z &.
}
\end{equation}

\begin{lemma}\label{lemma2.0.1}  \label{prop2.2}
We use the notations as above. Let $\sigma$ be the section of the projection $\pi: \hat{L}\rightarrow Z$ induced by the canonical map $\EE\rightarrow \OO_Z$. Denote by $C: =\sigma(Z)$ the image of the section $\sigma$. Let $R:=f^{-1}(D)_{red}$ be closed subscheme with reduced structure on the closed subset $f^{-1}(D)$ of $X$ of We have that:
\begin{enumerate}[(i)]
\item the line bundle $\OO_{\hat{L}}(X)$ associated to the divisor $X$ on $\hat{L}$ is isomorphic to $\pi^*\mathcal{L}^k\otimes\OO_{\PP(\EE)}(k)$;
\item $\OO_{X}(R)= \OO_{\hat{L}}(C)|_X=f^*\LL;$
\item the normal sheaf $N_{X/\hat{L}}$ of $X$ in $\hat{L}$ is isomorphic to $f^*\mathcal{L}^{k}$;
\item the canonical sheaf $ \kappa_X$ of $X$ is isomorphic to $f^*(\kappa_Z\otimes\mathcal{L}^{k-1})$.
\end{enumerate}

\end{lemma}

\begin{proof} \label{2.0.5}
The line bundle $\OO_{\hat{L}}(X)$ can be written as 
\begin{equation}\label{eqXdiv}
\OO_{\PP(\EE)}(d)\otimes \pi^*\mathscr{M},
\end{equation} 
where $\mathscr{M}\in \Pic(Z)$. Suppose that $\xi$ is a fiber of the projection $\pi$. It follows from the defining equation (\ref{eqforvl}) of $X$ that the intersection number $[X]\cdot [\xi]$ is equal to $k$. Hence $d=k$. Notice that $\OO_{\PP(\EE)}(1)|_C$ is trivial since the section $C$ corresponds to the surjection $\EE\rightarrow \OO_Z$. Thus we have
\[
\sigma^*(\OO_{\hat{L}}(X)|_C)=\mathscr{M}.
\]
Let us denote $\tau\in \Gamma(\hat{L}, \pi^*\LL)$ the tautological section. The divisor $X$ is the zero set of the section $\tau^k-\pi^* s$. In fact, the image $C$ is the zero locus of $\tau$. Therefore, the defining equation $\tau^k-\pi^* s$ of $X$ restricts to $-\pi^*s$ on $C$. And the line bundle $\mathscr{M}$ is the line bundle associated to the section $s$, i.e. $\mathscr{M}=\LL^k$. It concludes the first assertion.

For the second assertion, the scheme-theoretic intersection $X\cdot C$
is $^k\sqrt{f^{-1}(D)}=R$ by studying the local defining equations of $X$ and $C$. Therefore $\OO_X(R)=f^*\LL$.

It follows from the first assertion that 
\[
N_{X/\hat{L}}=\OO_{\hat{L}}(X)|_{X}=f^*\mathcal{L}^k\otimes \OO_{\PP(\EE)}(k)|_{X}.
\] 
Note that the line bundle of the divisor $C$ on $\hat{L}$ is isomorphic to $\pi^*\mathcal{L}\otimes \OO_{\PP(\EE)}(1)$(cf. \cite[Proposition 2.6, \S V]{Har77}) and $\OO_X(C)=f^*\LL$. We obtain $\OO_{\PP(\EE)}(1)|_X=\OO_X$. Therefore, the normal bundle $N_{X/\hat{L}}$ is isomorphic to 
$f^*\mathcal{L}^k$.

The last assertion is due to the standard Hurwitz formula and notice that the ramification divisor of the map $f$ is $(k-1)\cdot R$.

\end{proof}




\section{Deformation Theory and Infinitesimal Torelli Theorem}

In this section, we show that the deformations of some automorphisms of a cyclic coverings are unobstructed, cf. Theorem \ref{thmunobs}. We start this section by recalling the theory of equivariant deformations.

Let $X$ be a smooth and proper scheme over a field $k$. Assume that $G$ is a finite group embedded into the automorphism group $\Aut_k(X)$, i.e. \[\iota : G\hookrightarrow \Aut_k(X).\] Denote by $\mathscr{C}_k$ the category of Artinian local k-algebras with residue field $k$. An infinitesimal deformation of $(X, \iota)$ over an Artinian local k-algebra $A$ is a triple $(\mathcal{X}, \tilde{\iota}, \psi)$ consisting of a scheme $\mathcal{X}$ flat and proper over $A$, together with an injective group homomorphism 
\[
\tilde{{\iota}}: G\hookrightarrow \Aut_A(\mathcal{X})
\] 
and a $G$-equivariant isomorphism 
\[
\psi: \mathcal{X}\times_{\Spec(A)} \Spec(k) \rightarrow X.
\]  
Two infinitesimal deformations $(\mathcal{X}, \tilde{\iota}, \psi)$ and $(\mathcal{X}', \tilde{\iota}', \psi')$ are called isomorphic if there exists a $G$-equivariant isomorphism 
\[
\Phi: \mathcal{X}\rightarrow \mathcal{X}'
\]  
over $\Spec(A)$ that induces the identity on the closed fiber $X$.

\begin{definition}
With the notations as above, the equivairant deformation functor \[\Def_X^{G}: \mathscr{C}_k\rightarrow \Sets\] assigns each $A\in \mathscr{C}_k$ to the set $\Def^G_X(A)$ of isomorphism classes of infinitesimal deformations of $(X, \iota)$ over $A$.
\end{definition}

\begin{definition}
Suppose that $F$ and $H$ are covariant functors from $\mathscr{C}_k$ to $\textbf{Sets}$. A morphism $F\rightarrow H$ is called smooth if for any surjection $B\rightarrow A$ in $\mathscr{C}_k$, the map 
\[
F(B)\rightarrow F(A)\times_{H(A)} H(B)
\] 
is surjective. 
\end{definition} 
\begin{remark}
A single functor $F$ is called smooth if the canonical morphism $F\rightarrow G$, where $G=\{\text{one element}\}$, is smooth.
\end{remark}

The covariant functors from $\mathscr{C}_k$ to $\Sets$ are called \textit{functors of Artin rings} in \cite{Sch68}. We refer to the following proposition which is used to prove the smoothness of morphisms of \textit{functors of Artin rings}.

\begin{proposition}\label{3.1}\cite[Proposition 2.3.6]{Ser06}
Let $\mathscr{C}_k$ be the category of Artinian local k-algebras. Suppose that $F$ (resp. $H$) is the functor of Artin rings having a semiuniversal formal element and an obstruction space $obs(F)$ (resp. $obs(H)$). Let $k[\epsilon]\in \mathscr{C}_k$ be the dual number and $t_F:=F(k[\epsilon])$ be the space of first-order deformations. If a morphism $h: F\longrightarrow H$ satisfies the following two conditions:
\begin{enumerate}
\item the tangent map $dh: t_F\rightarrow t_H$ is surjective;
\item the obstruction map $\delta: obs(F)\rightarrow obs(H)$ is injective,
\end{enumerate}
then $h$ is smooth.
\end{proposition} 

In order to describe the tangent space and obstruction space of the equivariant deformation functor, we briefly recall Grothendieck's equivariant cohomology theory \cite{Gro1957}: Consider the covering map $\pi: X\rightarrow Y:=X/G$. There are two left exact functor defined on the category of $(G, \OO_X)$-modules, namely $\pi_*^G$ and $\Gamma^G(X, \bullet)$. For any sheaf of $G$-module $\mathscr{F}$ on $X$, $\pi_*^G(\mathscr{F})$ is the sheaf of $\OO_Y$-modules
\[
V\mapsto \Gamma(\pi^{-1}V, \mathscr{F})^G, \text{~where~} V \text{~is open in~} Y,
\]
and $\Gamma^G(X, \mathscr{F})$ is the group $\Gamma(X, \mathscr{F})^G$. We denote by $R^q\pi_*^G(\mathscr{F})$(resp. $H^q(X, G, \mathscr{F})$) the right derived functors of $\pi_*^G$(resp. $\Gamma^G(X, \bullet)$). There exist two Leray spectral sequences convergent to $H^{\bullet}(X, G, \mathscr{F})$
\begin{align*}
^I E^{p,q}_2&=H^p(Y, R^q\pi_*^G(\mathscr{F}))\Rightarrow H^{p+q}(X, G, \mathscr{F}),\\
^{II} E^{p,q}_2&=H^p(G, H^q(X, \mathscr{F}))\Rightarrow H^{p+q}(X, G, \mathscr{F}).
\end{align*}

Suppose that $X$ is a non-singular complete curve over $k$ with a finite group $G$ acts on it. Let $\Theta_X$ be the sheaf of tangent bundle of $X$. Bertin and M\'ezard showed in \cite{BM00} that the tangent space(resp. obstruction space) of the equivariant deformation functor $\Def_X^{G}$ is $H^1(X, G, \Theta_X)$(resp. $H^2(X, G, \Theta_X)$). The result naturally generalize to smooth proper $k$-schemes of higher dimensions.

Assume that the characteristic of the field $k$ does not divide the order of the group $G$. Then for any coherent $G$-sheaf $\mathscr{F}$, the cohomology group $H^q(X, \mathscr{F})$ is a $k$-vector space and the group cohomology $H^*(G, H^q(X, \mathscr{F}))$ vanishes if $*\neq 0$. Therefore, in this situation, the second Leray spectral sequence degenerates, and we concludes that $H^1(X, \Theta_X)^G$(resp. $H^2(X, \Theta_X)^G$) is the tangent space(resp. obstruction space) of the equivariant deformation functor $\Def_X^G$.

In the rest of the section, we may assume that the base field is $\CC$. Consider the natural forgetful map
\begin{equation}\label{forgeth}
h: \Def_X^G \rightarrow \Def_X
\end{equation}
assigning an infinitesimal deformation $(\mathcal{X},  \tilde{\iota}, \psi)$ over $A$ to the underlying infinitesimal deformation $\mathcal{X}$ over $A$.
Then the associated tangent map 
\[ 
dh: \Def^G_X(k[\epsilon]) =H^1(X, \Theta_X)^G\rightarrow H^1(X, \Theta_X)=\Def_X(k[\epsilon])\] and the obstruction map
\[
\delta: obs(\Def^G_X)=H^2(X, \Theta_X)^G\rightarrow H^2(X, \Theta_X)=obs(\Def_X)
\]
are both natural inclusions. 
\begin{proposition} \label{prophsm}
Use the same notations as above. Suppose that $n$ is the dimension of $X$ and the cup product 
\begin{equation}\label{3.4.1}
\lambda_p: H^1(X, \Theta_X)\rightarrow \Hom(H^{n-p}(X, \Omega_X^p), H^{n-p+1}(X, \Omega_X^{p-1}))
\end{equation} is injective for some $p$. If the group $G$ acts trivially on $H^n(X, \mathbb{C})$, then the forgetful functor $h$ in (\ref{forgeth}) is smooth.
\end{proposition}

\begin{proof}
We set $F=\Def_X^G$ and $H=\Def_X$ in Proposition \ref{3.1}. Then the condition (2) of Proposition \ref{3.1} is automatically satisfied. In order to verify the condition (1), we consider the diagram 
\begin{equation}\label{eqr-inf-tor}
\xymatrix{
H^1(X, \Theta_X)^G\ar[r] \ar[d] &\Hom(H^{n-p}(X, \Omega_X^p), H^{n-p+1}(X, \Omega_X^{p-1}))^G \ar@{=}[d]\\
H^1(X, \Theta_X)\ar[r] &\Hom(H^{n-p}(X, \Omega_X^p),H^{n-p+1}(X, \Omega_X^{p-1})).
}
\end{equation}
The right vertical identity follows from the assumption that the action of $G$ on $H^n(X, \mathbb{C})$ is trivial. Therefore, the injectivity of $\lambda_p$ implies that \[H^1(X, \Theta_X)^G=H^1(X, \Theta_X).\] Therefore, the forgetful functor $h$ is smooth.
\end{proof}




The  infinitesimal Torelli theorem of cyclic coverings over projective spaces is explored by J. Wehler.
\begin{theorem}(\cite[Theorem 4.8]{Weh86})\label{thm3.5}
Let $X$ be a smooth cyclic covering of $\PP^n_{\CC}$ of dimension $n\geq 2$. Then the cup product \[\lambda_p: H^1(X, \Theta_X)\rightarrow \Hom(H^{n-p}(X, \Omega_X^p), H^{n-p+1}(X, \Omega_X^{p-1}))\] is injective for some $p$ with the exceptions
\begin{itemize}
\item $X$ is a 3-fold covering of $\PP^2$ branched along a cubic curve;
\item $X$ is a 2-fold covering of $\PP^2$ branched along a quartic curve.
\end{itemize} 
\end{theorem}

In Proposition \ref{rk3.8}, we will prove that the deformation functor $\Def_X$ is smooth. First of all, we briefly recall some notions of deformations of morphisms. Let $f: X\rightarrow Z$ be a morphism between $k$-schemes, and let $A$ be an artinian local $k$-algebra. An infinitesimal deformation $(\mathcal{X}, F)$ over $A$ of the morphism $f$ is the cartesian diagrams 
\begin{equation}\label{3.6.1}
\xymatrix{
X \ar[d]^f \ar[r] &\mathcal{X} \ar[d]^F\\
Z \ar[d] \ar[r] & Z\times \Spec(A) \ar[d]^p\\
\Spec(k)\ar[r] & \Spec(A),
}
\end{equation}
where $p\circ F$ is flat. Two deformations $(\mathcal{X}, F)$ and $(\mathcal{X}', G)$ are isomorphic if there exists an isomorphism $\psi : \mathcal{X}\overset{\sim}{\rightarrow} \mathcal{X}'$ such that $G\circ \psi=F$ and the restriction of $\psi$ to the closed fiber $X$ gives the identity $\Id_X$. One defines the deformation functor of the morphism $f$ 
\[
\Def_{X/Z}(A)=\{\text{isomorphic classes of infinitesimal deformations of $f$ over $A$}\}.
\]
We call the functor $\Def_{X/Z}$ a local Hilbert functor, denoted by  $H^{X}_{Z}$, if $f$ and $F$ in the diagram (\ref{3.6.1}) are closed immersions. Let $\varrho : \Def_{X/Z}\rightarrow \Def_X$ and $\delta: H^{X}_{Z} \rightarrow \Def_X$ be the natural forgetful maps. \begin{theorem}(\cite[Theorem 3.9]{Weh86})\label{thm3.8}
Let $f: X\rightarrow \PP_{\CC}^n$ be a smooth $k$-fold cyclic covering of dimension $n\geq2$. Recall in the diagram (\ref{linebundleL}) that $\hat{L}$ denotes the projective bundle. If $X$ is not a $K3$-surface, then the forgetful maps $\varrho : \Def_{X/\PP^n}\rightarrow \Def_X$ and $\delta: H^{X}_{\hat{L}} \rightarrow \Def_X$ are both smooth.
\end{theorem}

The smoothness of the map $\varrho : \Def_{X/\PP^n}\rightarrow \Def_X$ shows that any small deformation of $X$ is again a cyclic covering of $\PP^n$ except $K$3 surfaces.

\begin{proposition}\label{rk3.8}
Let $X$ be a smooth $k$-fold cyclic covering of $\PP_{\mathbb{C}}^n$. The deformation functor $\Def_X$ is smooth.
\end{proposition}
\begin{proof}
If $n=1$ then it is obvious that $\Def_X$ is smooth. Therefore, we can assume that $n$ is at least $2$. The obstruction space of the local Hilbert functor $H^X_{\hat{L}}$ is isomorphic to $H^1(X, N_{X/\hat{L}})$. By Proposition \ref{prop2.2} (iii), we have
\[
H^1(X, N_{X/\hat{L}})=H^1(X, f^*\LL^k)=H^1(\PP^n_{\mathbb{C}}, \LL^k\otimes f_*\OO_X)=0.
\]
Hence $H^X_{\hat{L}}$ is smooth. If $X$ is not a $K$3 surface, the deformation functor $\Def_X$ is smooth since the forgetful morphism $\delta$ is smooth. When $X$ is a $K$3 surface, the smoothness of $\Def_X$ is well-known.
\end{proof}

\begin{theorem} \label{thmunobs}
Suppose that $X$ is a smooth $k$-fold cyclic covering of $\PP^n_{\mathbb{C}}$. Let $G$ be a finite subgroup of the automorphisms $\Aut(X)$. If $G$ acts on $H^n(X,\mathbb{C})$ trivially, then the equivariant deformations of $(X, G)$ is unobstructed, i.e. the deformation functor $\Def_X^G$ is smooth except
\begin{itemize}
\item $X$ is a 3-fold covering of $\PP^2$ branched along a cubic curve;
\item $X$ is a 2-fold covering of $\PP^2$ branched along a quartic curve;
\end{itemize}
Furthermore, under the above hypotheses, we have $\Def_X^G=\Def_X$.
\end{theorem}
\begin{proof}
By Proposition \ref{prophsm} and Theorem \ref{thm3.5}, we conclude that the forgetful functor $h:\Def^G_X\rightarrow \Def_X$ is smooth in our case. We prove the deformation functor $\Def_X$ is smooth in Proposition \ref{rk3.8}, then it follows that $\Def^G_X$ is also smooth. By the diagram (\ref{eqr-inf-tor}), the differential map between the first-order deformations
\[
dh: \Def^G_X(\mathbb{C}[\epsilon]) =H^1(X, \Theta_X)^G \rightarrow H^1(X, \Theta_X)=\Def_X(\mathbb{C}[\epsilon])
\]
is an identity. It implies that $\Def^G_X=\Def_X$.
\end{proof}
\begin{remark} The theorem amounts to say that an automorphism in $G$ can be deformed to an automorphism of any deformation of $X$. Instead of using the equivariant deformation theory, the second author provides an alternative view point to prove this theorem from the variational Hodge conjectures for graph cycles, cf. \cite[Corollary 3.3]{Pan16}.
\end{remark}
\section{Automorphisms of Finite Cyclic Coverings}

\begin{lemma}\label{lemma4.2} \label{prop4.3}
Let $f: X\rightarrow \PP^n$ be a smooth $k$-fold cyclic covering branched along a smooth hypersurface $D$, and let $\sigma$ be an automorphism of $X$. Denote by $\Cov(X/\PP^n)$ the group of covering transformations. Suppose that the following two conditions

\begin{itemize}
\item the line bundle $\sigma^*f^*\OO_{\PP^n}(1)$ is isomorphic to $f^*\OO_{\PP^n}(1), \forall \sigma\in \Aut(X)$;
\item $\dim H^0(X, f^*\OO_{\PP^n}(1))=n+1$
\end{itemize}
are satisfied. Then there is a short exact sequence 
\begin{equation}\label{4.5.1}
1\longrightarrow \Cov(X/\PP^n)\longrightarrow \Aut(X)\longrightarrow \Aut_L(D)\longrightarrow 1.
\end{equation}
Here the group $\Aut_L(D)$ indicates the group of the linear automorphisms of $D\subseteq \PP^n$.
\end{lemma}

\begin{proof}

Under our assumptions, every automorphism $\sigma$ induces a unique automorphism $\mu$ of $\PP^n$ such that $f\circ \sigma=\mu\circ f$. In fact, the morphism $f$ gives rise to global sections $s_i=f^*x_i$ of $f^*\OO_{\PP^n}(1)$ for $i=\{0,1,\cdots , n\}$. If $\dim H^0(X, f^*\OO_{\PP^n}(1))=n+1,$ then the set $\{s_0,\cdots ,s_n\}$ forms a basis of the complete linear system $|f^*\OO_{\PP^n}(1)|$. If $\sigma^*f^*\OO_{\PP^n}(1)$ is isomorphic to the line bundle $f^*\OO_{\PP^n}(1)$, then it follows that $\sigma^*s_i=\sum\limits_{j=0}^{n}\alpha_{ij} s_j$. Hence, the matrix $(\alpha_{ij})_{0\leq i,j\leq n}$ gives the desired automorphism 
\[
\mu\big(\big[ X_0: X_1,\cdots, X_n \big] \big)=\big[ \sum\limits_{i=0}^{n}\alpha_{0i}X_i,\cdots, \sum\limits_{i=0}^{n}\alpha_{ni}X_i \big]
\] in $\Aut(\PP^n)$.

Through the above argument, the linear automorphism $\mu$ apparently preserves the branch locus $D$. We thus obtain a homomorphism 
\begin{align*}
\Aut(X)&\rightarrow \Aut_L(D)\\
\sigma &\rightarrow \mu |_D
\end{align*} 
with kernel $\Cov(X/\PP^n)$. The homomorphism $\Aut(X)\rightarrow \Aut_L(D)$ is surjective by the definition of the cyclic coverings.
\end{proof}

\begin{lemma} \label{lemma4.3}
Let $X$ be smooth k-fold cyclic covering $f: X\rightarrow \PP^n$ branched along a smooth hypersurface $D$. If $X$ is not a hypersurfaces in $\PP^{n+1}$, then \[ \Dim~H^0(X, f^*\OO_{\PP^n}(1))=n+1.\]
\end{lemma}

\begin{proof} 

Assume that $\LL$ is the line bundle $\OO_{\PP^n}(m)$ such that $\LL^k=\OO_{\PP^n}(D)$. $X$ is not a hypersurface if and only if $m$ is great than $1$. Since $f$ is a finite morphism, we have
\begin{align*}
H^0(X, f^*\OO_{\PP^n}(1))&= H^0(\mathbb{P}^n, \mathcal{O}_{\mathbb{P}^n}(1)\otimes f_*\mathcal{O}_X)\\
&=\bigoplus\limits_{i=0}^{k-1}H^0(\PP^n, \OO_{\PP^n}(1)\otimes \mathcal{L}^{-i}). \nonumber
\end{align*}
Therefore we obtain $\Dim H^0(X, \OO_X(1))=n+1$ when $m > 1$.
\end{proof}


From now on take $K=\mathbb{C}$. In the following proposition, we investigate the smooth cyclic coverings who satisfy the first condition in Lemma \ref{lemma4.2}.
\begin{proposition}\label{prop4.6}
Let $X$ be a smooth k-fold covering $f: X\rightarrow \PP^n$ branched along a smooth hypersurface $D$. If one of the following conditions hold:
\begin{enumerate}
\item$\dim X\geq 4$;
\item $\dim X=3$ and the branch locus $D$ is a smooth surface in $\PP^3$ with $\deg(D)\geq 4$,
\end{enumerate} then the Picard group of $X$ is generated by the ample line bundle $f^*\OO_{\PP^n}(1)$. In particular, $\sigma^*f^*\OO_{\PP^n}(1)$ is isomorphic to $f^*\OO_{\PP^n}(1)$ for any automorphism $\sigma$ of $X$.

\end{proposition}

\begin{proof}
We denote by $\LL$ the line bundle the as same notation in Definition \ref{def2.1}, which defines a cyclic covering. The ramification divisor $R$ is an integral ample divisor since $\OO_X(R)\simeq f^*\LL$, see Lemma \ref{lemma2.0.1}.
\begin{enumerate}[(i)]
\item Suppose that the dimension of $X$ is at least 4. The Lefschetz hyperplane section theorem gives the isomorphism
\[
\mu^*: \Pic(X)\simeq \Pic(R).
\] 
induced by the inclusion $\mu: R\hookrightarrow X$. In fact, we have the following natural commutative diagram
\begin{equation} \label{4.7.2}
\xymatrix{
\Pic(\PP^n)\ar[r]^{\nu^*} \ar[d]^{f^*}&\Pic(D) \ar[d]^{f|_B^*}\\
\Pic(X)\ar[r]^{\mu^*}&\Pic(B).
}
\end{equation}
Again by the Lefschetz hyperplane section theorem, the homomorphism $\nu^*$ induced by the natural inclusion is an isomorphism. And $f|_R: R\rightarrow D$ is an isomorphism of varieties. Then it follows that
\[
\Pic(X)=\mathbb{Z}\langle f^*\OO_{\PP^n}(1)\rangle.
\]

\item Suppose that the dimension of $X$ is 3. It is easily to see that
\[
H^j(X, \OO_X)\simeq H^j(\PP^n, f_*\OO_X)=\bigoplus\limits_{i=0}^{k-1}H^j(\PP^3, \LL^{-i})=0~\text{for}~ j=1,2.
\]
Therefore the cycle class map $c_1: \Pic(X)\rightarrow H^2(X, \mathbb{Z})$ is an isomorphism for any smooth cyclic covering $X$.

Let $D$ be a very general smooth surface in $\PP^3$ with $\text{deg}(D)\geq 4$, the Noether-Lefschetz Theorem yields an isomorphism $\nu^*: \Pic(\PP^3)\rightarrow \Pic(D)$. Hence the induced map $\mu^*: \Pic(X)\rightarrow \Pic(R)$ is surjective since it admits an inverse section $f^*\circ \nu^{*-1}\circ f^*|_R^{-1}$, see the diagram (\ref{4.7.2}).

On the other hand, the restriction map $H^2(X, \mathbb{Z})\rightarrow H^2(R, \mathbb{Z})$ is injective by the Lefschetz hyperplane theorem, which implies $\mu^*:\Pic(X)\rightarrow \Pic(B)$ is injective. Therefore, we have 
\[
\Pic(X)\simeq \Pic(\PP^3)=\mathbb{Z}\langle\OO_{\PP^n}(1)\rangle.
\] 
for a very general $X$. 

Since the second cohomolgy group $H^2(-, \mathbb{Z})$ is a deformation invariant. The assertion holds for any cyclic covering $X$ of $\PP^3$ branched along a smooth surface $D$ with $\deg D\geq 4$.
\end{enumerate}
\end{proof}

\begin{theorem}[Finiteness of Automorphisms]\label{finiteauto}
Let $X$ be a smooth finite cyclic covering of $\PP^n_{\mathbb{C}}$  branched along a smooth hypersurface $D$ of degree $d$. Assume that $X$ is not a quadric hypersurface and the dimension of $X$ is at least 3. Then the automorphism group $\Aut (X)$ is finite. Moreover, if $X$ is very general, the automorphism group $\Aut(X)=\Cov(X/\PP^n)$. 
\end{theorem}

\begin{proof}

Note that $X$ is a quadric hypersurface if and only if $d=2$. In the rest we may assume that $d$ is at least 3. 
\begin{itemize}
\item $X$ is a smooth hypersurface in $\PP^{n+1}$ with $n\geq 3$. Poonen proves that for a smooth hypersurface $Y\subset \PP^m$ of degree $l$, the linear automorphism group $\Aut_L(Y)$ is finite if $m\geq 2$ and $l\geq 3$, and $\Aut(Y)=\Aut_{L}(Y)$ if $m\neq 2, l\neq3$ or $m\neq 3, l\neq 4$(\cite[Theorem 1.1, Theorem 1.3]{Poo05}). Therefore, in our case $\Aut(X)=\Aut_{L}(X)$. In particular, the automorphism group $\Aut(X)$ is finite.

\item $X$ is not a hypersurface. Then the two conditions in Lemma \ref{lemma4.2} hold for $X$ by Lemma \ref{lemma4.3} and Proposition \ref{prop4.6}. Therefore, by the exact sequence (\ref{4.5.1}) in Lemma \ref{lemma4.2} and the above Poonen's result, we know $\Aut(X)$ is finite.

For a very general $X$, we have $\Aut(X)=\Cov(X/\PP^n)$ due to the classical result in \cite{MM}, which shows $\Aut_L(D)$=\{id\} for a very general hypersurface $D$ with $\deg(D)\geq 3$ and $\dim D\geq 2$.
\end{itemize}
\end{proof}

\begin{lemma}\label{Covfaith}
Let $X$ be a smooth $k$-fold cyclic covering of a projective space $\PP^n_\CC$ branched along a smooth divisor $D$. The representation on the group of covering transformations
\[
\varphi|_{\Cov(X/\PP^n)}: \Cov(X/\PP^n)\rightarrow\mathrm{GL}(\textnormal{H}^n(X, \mathbb{C}))
\]
is faithful. 
\end{lemma}

\begin{proof}

Let $\LL$ be the line bundle on $\PP^n$ such that $\LL^k=\OO_{\PP^n}(D)$. There is a decomposition of the sheaf of differential forms(see \cite[Lemma 3.16]{EV92})
\begin{equation}\label{decom-diff}
f_*\Omega^q_X=\Omega^q_{\PP^n}\oplus\bigoplus\limits_{i=1}^{k-1} \Omega^q_{\PP^n}(\log D)\otimes\mathcal{L}^{-i}.
\end{equation}

If $g$ is a generator of the group $\Cov(X/\PP^n)$, for any integer $m$ the automorphism $g^m$ acts on $\Omega^q_{\PP^r}(\log D)\otimes \mathcal{L}^{-i}$ by multiplying by $\varrho^{mi}$. It splits the Hodge cohomology group $H^p(X, \Omega^q_X)$ into $\varrho^{mi}$-eigenvalue subspaces
\begin{equation}
H^p(X, \Omega^q_X)=H^p(\PP^n, \Omega^q_{\PP^n})\oplus\bigoplus\limits_{i=1}^{k-1} H^p(\PP^n, \Omega^q_{\PP^n}(\log D)\otimes \mathcal{L}^{-i}).
\end{equation}
If $\psi(g^m)$ is $\Id_{H^{n}(X,\mathbb{C})}$, then $m$ is equal to $0$ modulo $k$. Hence the representation is faithful.

\end{proof}

\begin{proposition} \label{propfaithsurface}
Let $f:X\rightarrow \PP^2$ be a smooth finite cyclic covering over $\PP^2_{\CC}$. If $X$ is not a quadric surface, then the action of $\Aut(X)$ on $H^2(X,\mathbb{C})$ is faithful.
\end{proposition}
\begin{proof}

Suppose that $X$ is a $k$-fold covering with the line bundle $\LL:=\OO_{\PP^2}(m)$. Recall Proposition \ref{prop2.2} that the canonical linke bundle $\kappa_X$ is 
\begin{equation}\label{eqcanonical}
\kappa_X=f^*(\kappa_{\PP^2}\otimes \LL^{k-1}).
\end{equation}

If $\kappa_X$ is trivial, then $X$ is a $K3$ surface. The conclusion is well known.

If $\kappa_X$ is ample, it is isomorphic to $f^*\OO_{\PP^2}(d)$ for some positive integer $d$. The morphism $f$ gives rise to global sections $\{s_i=f^*x_i~|~i=0,1,2\}$, which generate the line bundle $f^*\OO_{\PP^2}(1)$. Let $N=\binom{d+2}{d}-1$. Consider the d-uple Veronese embedding $\nu_d: \PP^2\hookrightarrow \PP^N$. The composition $h: X\xrightarrow{f} \PP^2\hookrightarrow \PP^N$ associates to the canonical line bundle $\kappa_X$ with globally generated sections $\{s_0^{d}, \ldots , s_2^{d}\}$. Assume that an automorphism $\sigma\in \Aut(X)$ acts trivially on $H^2(X, \mathbb{C})$. Hence we have $\sigma^*=\Id |_{H^0(X, \kappa_X)}$ and $\sigma^*h^*\OO_{\PP^N}(1)=\sigma ^*\kappa_X=h^*\OO_{\PP^N}(1)$. Therefore, $ h\circ \sigma =h$. It induces the following diagram 
\[\xymatrix{ X \ar[dd]_{\sigma} \ar[dr]_{f} \ar[drr]^h \\
& \PP^2 \ar@{^(->}[r] &\PP^N \\
X \ar[ru]^f \ar[urr]_h} \]
Therefore, the automorphism $\sigma$ lies in $\Cov(X/\PP^2)$. By Lemma \ref{Covfaith}, we obtain $g=\Id_X$.

If $\kappa_X$ is anti-ample, i.e., $X$ is a Fano surface, the possible cases are

\begin{itemize}\label{fanotype}
\item $(m, k)=(1,2)$, $X$ is a quadric surface in $\PP^3$;
\item $(m, k)=(2,2)$, $X$ is a 2-fold covering branched over a quartic curve;
\item $(m, k)=(1,3)$, $X$ is a cubic surface in $\PP^3$.
\end{itemize}

The last two cases are del Pezzo surfaces with degree $2$ and $3$ respectively.
Hence, they are blowups of projective planes along $7$ and $6$ points in general position respectively. Denote the blowup by $\tau: X\rightarrow \PP^2$ with exceptional divisor $E_i$. If $\sigma^*=\Id$ on $H^2(X,\mathbb{C})$, then $[\sigma(E_i)]\cdot [E_i]=-1$. Hence $\sigma(E_i)=E_i$ for all $i$. And the automrophism $\sigma$ arises an automorphism $\rho\in \Aut(\PP^2)$ with $\tau \circ \sigma=\rho \circ \tau$ and $\rho$ fixes more than $4$ points in general position. Then It follows that $\rho=\Id_{\PP^2}$ and $\sigma=\Id_X$.
\end{proof}

Now we are able to give the answer to the question \ref{question}\\
 in the introduction. 

\begin{theorem}\label{autofaith}
Let $f: X\rightarrow \PP^n_{\CC}$ be a smooth $k$-fold cyclic covering over the complex numbers with $n\geq 2$. Suppose that $X$ is not a quadric hypersurface. Then the natural representation
\[\varphi: \Aut(X)\longrightarrow \mathrm{GL}(\Ho^n(X, \mathbb{C}))\]
is faithful.
\end{theorem}

\begin{proof}
If the dimension of $X$ is $2$, it is a conclusion of Proposition \ref{propfaithsurface}. In the following the dimension of $X$ is at least $3$. 

Let $g$ be any automorphism of $X$ such that $\varphi(g)=\Id$, and let $G$ be the cyclic group generated by $g$. It follows from Theorem \ref{finiteauto} that $G$ is a finite group and Theorem \ref{thmunobs} that the action of $G$ on $X$ can be deformed to an action of $G$ on a very general deformation $Y$ of $X$ with $G\subset \Aut(Y)$. Recall that $Y$ remains a smooth cyclic covering. Hence it follows from Theorem \ref{finiteauto} that $\Aut(Y)=\Cov(Y/\PP^n)$. By the faithfulness of the representation $\varphi|_{\Cov(X/\PP^n)}$, we conclude that the group $G$ is trivial. Therefore, a very general deformation of the automorphism $g$ of $X$ is the identity map. It implies that $g=\Id_X$ by specialization.

\end{proof}

\section{Hodge Decomposition for Finite Cyclic Coverings}

Let $f: \mathfrak{X}\rightarrow S$ be a smooth and proper morphism of schemes. The relative lgebraic de Rham cohomology is defined to be the sheaf of  hyper-cohomology of the algebraic de Rham complex $\Omega^{\bullet}_{\mathfrak{X}/S}$ 
\[
\mathbb{H}^k_{dR}(\mathfrak{X}/S):= \mathbb{R}^kf_*\Omega^{\bullet}_{\mathfrak{X}/S}.
\]
The first spectral sequence of the hypercohomology, which abuts to the relative de Rham cohomology,
\begin{equation}\label{5.0.2}
E^{p,q}_1= R^qf_*\Omega^p_{\mathfrak{X}/S}\Rightarrow \mathbb{H}^{p+q}(\mathfrak{X}, \Omega^{\bullet}_{\mathfrak{X}/S})
\end{equation}
is called the relative Hodge-de Rham spectral sequence.

If $S=\Spec\CC$, the classical Hodge decomposition 
\[
\bigoplus_{p+q=k}H^q(\mathfrak{X}, \Omega^p_{\mathfrak{X}/\CC})=H^k_{dR}(\mathfrak{X}/\CC)
\]
implies that the spectral sequence (\ref{5.0.2}) of $\mathfrak{X}$ degenerates at $E_1$. 

In general, the degeneration of the spectral sequence \ref{5.0.2} is an interesting but hard question. In \cite{SGA7-II}, Deligne proved the case of relative complete intersections.

Let $\mathcal{E}$ be a locally free sheaf of rank $r+1$ over a scheme $S$, let $p: \mathbb{P}(\mathcal{E})\rightarrow S$ be the relative projective bundle, and let $(a_1, \ldots, a_d)$ be a $d$-tuple of positive integers. One can define a relative complete intersection $Y$ in $\mathbb{P}(\mathcal{E})$ of multi-degree $(a_1, \ldots, a_d)$. 

\begin{theorem}\cite[Expos\'e XI, Theorem 1.5]{SGA7-II}\label{thm5.1}
Let $f: Y\rightarrow S$ be the relative complete intersection of relative dimension $n-d$ as above, and let $\eta\in H^0(S, R^1f_*\Omega^1_{Y/S})$ be the canonical section induced by the first Chern class of $\OO(1)$. Assume $f: Y\rightarrow S$ is smooth, then we have
\begin{enumerate}
\item The coherent sheaf $R^qf_*\Omega^p_{Y/S}$ is locally free, and any base change map is isomorphic.
\item The associated relative Hodge-de Rham spectral sequence degenerates at $E_1$.
\item  a) $R^qf_*\Omega^p_{Y/S}=0,$~ if $p\neq q$ and $p+q\neq n$;

b) If $2p< n$, then the global section $\eta^i: \OO_S\rightarrow R^pf_*\Omega^p_{Y/S}$ induced by $\eta^i$ is an isomorphism; 

c) If $n<2p\leq 2n$, let the section $\xi_p\in H^0(S, R^pf_*\Omega^p_{Y/S})$ dual to the section $\eta^{n-p}\in H^0(S, R^{n-p}f_*\Omega^{n-p}_{Y/S})$, then $\eta^p=(\displaystyle{\prod_{p=1}^d a_p) \xi_p}$;

d) If $n=2m$, then $\eta^m\neq 0$ and the quotient $R^mf_*\Omega^m_{Y/S}/\OO_S\cdot \eta^m$ is locally free.
\end{enumerate}
\end{theorem}

\begin{remark}
The assertions $(1)$ and $(2)$ are deduced from $(3)$ 
\end{remark}
In this section, the main goal is to prove the relative Hodge-de Rham spectral sequence of a smooth family of cyclic coverings of projective spaces degenerates at $E_1$. We will prove statements analogous to the ones in Theorem \ref{thm5.1} for family of cyclic coverings, see Theorem \ref{thm5.4}.

\begin{definition}\label{def5.3}
Let $p: \PP^n_S\rightarrow S$ be a projective bundle over a scheme $S$. Let $\mathcal{L}$ be a positive line bundle $\OO_{\PP^n_S}(r)$. Let $s\in \Hom(\OO_{\PP^n_S}, \mathcal{L}^k)$ be a generic section for some positive integer $k$ such that the zero locus $D:= Z(s)$ is smooth family of hypersurfaces over $S$. As in Definition \ref{def2.1}, the $\OO_{\PP^n_S}$-algebra 
\[
\mathcal{A}=(\bigoplus\limits_{i=0}^{k-1}\mathcal{L}^{-i}).
\]
defined by the section $s$ induces a relative affine morphism
\begin{equation}\label{5.3.1}
\xymatrix{
\mathcal{X}:=\underline{\Spec}(\mathcal{A}) \ar[rr]^-{f} \ar[rd]_{\pi}&&\PP^n_S \ar[ld]^{p}\\
&S,
}
\end{equation}
which is called the family of $k$-fold cyclic coverings of $\PP^n_S$ over $S$ branched along $D$ 

\end{definition}

The family $\pi: \mathcal{X}\rightarrow S$ is flat since $D$ is flat over $S$ and then $f$ is also flat. Through the rest of this section, we assume that $\pi$ is smooth. The smoothness ensures that the integer $k$ is not divided by the characteristic of the residue field $\kappa(s)$ for all $s\in S$. In particular, suppose that $S$ is the spectrum of a discrete valuation ring. If $k$ is coprime to the characteristics of the residue field and the fraction field, then the family of $k$-fold cyclic coverings is smooth. If $S$ is an $F$-scheme where $k$ is less than the characteristic of $F$, the family is also smooth.

\begin{lemma}\label{lemma5.4}
With the same notations in Definition \ref{def5.3}. We have that
\begin{align}
R^if_*\Omega^j_{\mathcal{X}/S}&=0, ~i\neq 0; \label{5.5.1} \\
f_*\Omega^j_{\mathcal{X}/S}&=\Omega^j_{\PP^n_S/S}\oplus \bigoplus\limits_{\mu=1}^{k-1}\Omega^j_{\PP^n_S/S}(\log D)\otimes \LL^{-\mu}; \label{5.5.2}\\
R^i\pi_*\Omega^j_{\mathcal{X}/S}
&=R^ip_*(\Omega^j_{\PP^n_S/S}\oplus \bigoplus\limits_{\mu=1}^{k-1}\Omega^j_{\PP^n_S/S}(\log D)\otimes \LL^{-\mu}).\label{5.4.3}
\end{align}
\end{lemma}

Let $S$ be a scheme,$X$ a smooth $S$-scheme, and $D$ a closed subscheme of $X$. Let $j: U:=X\slash D\hookrightarrow X$ be the open immersion. The \textit{logarithmic de Rham complex}:
\[
\Omega^{\bullet}_{X/S}(\log D)\
\]
can be defined as  a subcomplex of $j_*\Omega_{U/S}^{\bullet}$ if $D$ is a divisor with normal crossing relative to $S$. Here a normal crossing divisor relative to $S$ means locally for the \'etale topology, the pair $(X, D)$ is induced by the standard affine space $\mathbb{A}^n_S=\Spec S[t_1,\ldots, t_n]$ and the divisor $V(t_1\cdots t_r)$. For more details, see \cite[p. 137]{Hodge}.

\begin{proof}[Proof of Lemma 5.4]
The absolute version of the above decompositions has been proved, see \cite[Lemma 3.16 d)]{EV92}. For the sake of completeness, we show that the proof could be carried out in the relative version.

The morphism $f$ is finite. Hence, the first assertion follows. Let $\mathbb{A}^n_S$ be an affine open subset of $\PP^n_S$, and let $U\subset \XX$ be the inverse image $f^{-1}(\mathbb{A}^n_S)$. Denote by $s'$ the local defining equation of the branched locus $D\cap \mathbb{A}^n_S$ on $\mathbb{A}^n_S$. Let the tuple $\{s', x_1,\cdots , x_{n-1}\}$ be a local coordinate system of $\mathbb{A}^n_S$, which induces a basis $\{ ds',dx_1,\cdots, dx_{n-1}\}$ of the sheaf of $1$-forms $\Omega^1_{\mathbb{A}^n_S/S}$. Then the sheaf of $\OO_{\mathbb{A}^n_S}$-module $\Omega^1_{\mathbb{A}^n_S/S}(\log D)$ is locally free of finite type with a basis $\{\frac{ds'}{s'}, dx_1,\cdots, dx_{n-1}\}$. Similarly, we have a local coordinate system $\{t', f^*x_1,\cdots, f^*x_{n-1}\}$ on $U$, where $t'$ is the restriction of the tautological section $t\in H^0(\XX, f^*\LL)$ to $U$. Denote by $B$ the zero locus of the section $t'$ in $U$. Then the associated sheaf of $\OO_U$-module $\Omega^1_{U/S}(\log B)$ is locally free of finite type with a basis $\{\frac{dt'}{t'}, f^*dx_1,\cdots, f^*dx_{n-1}\}$. Firstly, we claim the following relative Hurwitz's formula
\[
f^*\Omega^j_{\mathbb{A}^n_S/S}(\log D)=\Omega^j_{U/S}(\log B).
\]
In fact, recall the Definition \ref{def5.3} that we have an equation $t'^k-f^*(s')=0$ on $U$, cf. (\ref{eqforvl}), which implies $f^*\frac{ds'}{s'}=\frac{dt'^k}{t'^k}= k\cdot\frac{dt'}{t'}$. The integer $k$ can be inverted in the ring $\Gamma(S, \OO_S)$ since $k$ is not divided by the characteristic of the residue field of any point of $S$. Therefore, the differential form $f^*\frac{ds'}{s'}$ and sheaf $f^*\Omega^1_{\mathbb{A}^n_S/S}$ generate $\Omega^1_{U/S}(\log B)$. Then the entire relative Hurwitz's formula follows from the exterior products. 

Following the above assertion, we have a natural inclusion
\[
f_*\Omega^j_{U/S}\subset f_*\Omega^j_{U/S}(\log B)=\Omega^j_{\mathbb{A}^n_S/S}(\log D)\otimes f_*\OO_U=\bigoplus\limits^{k-1}_{i=0}\Omega^j_{\mathbb{A}^n_S/S}(\log D)\otimes \LL^{-i}.
\]
We claim that the subsheaf 
\[
\Omega^j_{\mathbb{A}^n_S/S}\oplus \bigoplus\limits^{k-1}_{i=1}\Omega^j_{\mathbb{A}^n_S/S}(\log D)\otimes \LL^{-i}\subseteq \bigoplus\limits^{k-1}_{i=0}\Omega^j_{\mathbb{A}^n_S/S}(\log D)\otimes \LL^{-i}
\]
is equal to $f_*\Omega^j_{U/S}$. In fact, a local section $\sigma$ of $\Omega^j_{\mathbb{A}^n_S/S}(\log D)\otimes \LL^{-i}$ can be written as $\sigma=\psi\cdot s'^i$ for some local section $\psi$ of $\Omega^j_{\mathbb{A}^n_S/S}(\log D)$ and the local generator $s'$ of the bundle $\LL^{-1}$. The section $\psi$ has the form
\[
\omega\wedge \frac{ds'}{s'} \text{~or~} \omega,
\] 
where the local section $\omega$ has no pole along $D$. Therefore, the pullback of the section $\sigma$ is given by
\begin{align*}
f^*\sigma&=k\cdot f^*\omega\wedge \frac{dt'}{t'}\cdot f^*s'^i=k\cdot f^*\omega \wedge \frac{dt'}{t'}\cdot t'^i \text{~or~} \\
f^*\sigma&=f^*\omega\cdot f^*s'^i=f^*\omega\cdot t'^i.
\end{align*}
Note that $\sigma$ lies in $f_*\Omega^j_{U/S}$ if and only if the differential form $f^*\sigma$ has no pole along the divisor $B$. Therefore, the local section $\sigma$ lies in $f_*\Omega^j_{U/S}$ if and only if $i\geq 1$ or $i=0, \psi=\omega$, i.e. $\sigma$ is a local section of $\Omega^j_{\mathbb{A}^n_S/S}\oplus \bigoplus\limits^{k-1}_{i=1}\Omega^j_{\mathbb{A}^n_S/S}(\log D)\otimes \LL^{-i}$. We obatin the second formula.

Using the diagram (\ref{5.3.1}), we obtain the Leray spectral sequence 
\begin{equation}\label{5.4.5}
E^{a,b}_2=R^ap_*R^bf_*\Omega^j_{\mathcal{X}/S}\Longrightarrow R^i\pi_*\Omega^j_{\mathcal{X}/S} ,~ a+b=i. 
\end{equation}
By the first assertion, we have $E^{a,b}_2=0$ unless $b=0$. Therefore, the spectral sequence (\ref{5.4.5}) degenerates and it follows from the second assertion that

\[R^i\pi_*\Omega^j_{\mathcal{X}/S}=E^{i,0}_\infty=E^{i,0}_2
=R^ip_*(\Omega^j_{\PP^n_S/S}\oplus \bigoplus\limits_{\mu=1}^{k-1}\Omega^j_{\PP^n_S/S}(\log D)\otimes \LL^{-\mu}).\]

\end{proof}

\begin{proposition}\label{prop5.5}
Let $\pi: \mathcal{X}\rightarrow S$ be the smooth family of $k$-fold cyclic coverings associated to the line bundle $\LL$ defined in Definition \ref{def5.3}. Then we have
\[
R^i\pi_*(\Omega^j_{\mathcal{X}/S}\otimes f^*\LL^{-m})=0,~i+j<n, m\geq 1. 
\]

\end{proposition}

\begin{proof}

The Leray spectral sequence argument as (\ref{5.4.5}) shows
\[
R^i\pi_*(\Omega^j_{X/S}\otimes f^*\LL^{-m})=R^ip_*(\Omega^j_{\PP^n_S/S}\otimes\LL^{-m}\oplus \bigoplus\limits_{\mu=1}^{k-1}\Omega^j_{\PP^n_S/S}(\log D)\otimes \LL^{-\mu}\otimes \LL^{-m}).
\]

By Bott's vanishing theorem,  we have $R^ip_*(\Omega^j_{\PP^n_S/S}\otimes\LL^{-m})=0$.. Therefore, it suffices to prove 
\begin{equation}\label{eqeqzero}
R^ip_*(\Omega^j_{\PP^n_S/S}(\log D)\otimes  \OO_{\PP^n_S}(r))=0 \text{~for~} i+j<n \text{~and~} r<0.
\end{equation}
Denote by $\mA$ the line bundle $\OO_{\PP^n_S}(r)$. Consider the short exact sequence of residue map on logarithmic differential forms
\begin{equation}\label{5.5.3}
0\rightarrow \Omega^j_{\PP^n_S/S}\rightarrow \Omega^j_{\PP^n_S/S}(\log D)\xrightarrow{res} \iota_*\Omega^{j-1}_{D/S}\rightarrow 0
\end{equation} where $\iota: D\hookrightarrow \PP^n_S$ is the natural inclusion. In order to prove (\ref{eqeqzero}), it suffices to show \begin{align*}
R^ip_*(\Omega^j_{\PP^n_S/S}\otimes \mA) &=0 \text{~for~} i+j<n,\\
\text{and}~ R^qp_* (\iota_*\Omega^p_{D/S}\otimes \mA)&=0 \text{~for~} q+p < \dim D.
\end{align*}
The first conclusion follows again by the Bott's vanishing theorem since the invertible sheaf $\mA$ is negative. We prove the second one in the following.

Let $d$ be the degree of the smooth divisor $D$. Consider a natural resolution
\[
0\rightarrow \OO_D(-p\cdot d)\rightarrow \Omega^1_{\PP^n_S/S}(-(p-1)\cdot d)|_{D}\rightarrow \cdots \rightarrow \Omega^{p}_{\PP^n_S/S}|_{D}\rightarrow \Omega^p_{D/S}\rightarrow 0.
\]
of the sheaf of differential forms $\Omega^p_{D/S}$. Tensoring the resolution with $\iota^*\mA$, we obtain a complex $\mathcal{K}^{\bullet}$ whose $a$-th term is $(\Omega^{a}_{\PP^n_S/S}(-(p-a)\cdot d)\otimes \mA)|_D.$
Then the hypercohomology spectral sequence for the complex $\mathcal{K}^{\bullet}$ is
\[
E^{a,b}_1:=R^bp_*(\Omega^{a}_{\PP^n_S/S}(-(p-a)\cdot d)\otimes \mA|_{D})\Rightarrow \mathbb{R}^{a+b}p_*(\mathcal{K}^{\bullet})=R^{a+b-p}p_*(\Omega^p_{D/S}\otimes \mA|_D))
\]
We claim that \begin{equation}\label{5.6.6}
R^bp_*(\Omega^a_{\PP^n_S/S}(-(p-a)\cdot d)\otimes \mA|_D)=0 \text{~for~} a+b < \dim D.
\end{equation}

For simplicity, denote by $\mathscr{L}$ the negative line bundle $\OO_{\PP^n_S/S}(-(p-a)\cdot d)\otimes \mA$, and consider the short exact sequence
\[
0\rightarrow \Omega^a_{\PP^n_S/S}(-d)\otimes \mathscr{L}\rightarrow \Omega^a_{\PP^n_S/S}\otimes \mathscr{L}\rightarrow \Omega^a_{\PP^n_S/S}\otimes \mathscr{L}|_D\rightarrow 0.
\]
It is easily to see the claim (\ref{5.6.6}) follows from the Bott's vanishing
\begin{align*}
R^bp_*(\Omega^a_{\PP^n_S/S}\otimes \mathscr{L})&=0 \text{~for~} a+b<n-1\\
R^{b+1}p_*(\Omega^a_{\PP^n_S/S}(-d)\otimes \mathscr{L})&=0 \text{~for~} a+b<n-1.
\end{align*}

\end{proof}

\begin{theorem}\label{thm5.4}
Recall that 
\[
\xymatrix{
\mathcal{X} \ar[rr]^-{f} \ar[rd]_{\pi}&&\PP^n_S \ar[ld]^{p}\\
&S,
}
\]
is a smooth family of $k$-fold cyclic coverings of $\PP^n_S$ branched along $D$. Let $\eta\in H^0(S, R^1p_*\Omega^1_{\PP^n_S/S})$ be global section induced by the first Chern class of $\OO(1)$ on $\PP^n_S/S$, and let $f^*\eta\in H^0(S, R^1\pi_*\Omega^1_{\XX/S})$ be global section induced by the first Chern class of $f^*\OO(1)$. We have
\begin{enumerate}
\item $R^i\pi_*\Omega^j_{\mathcal{X}/S}=0,$~ if $i\neq j$ and $i+j\neq n$,
\item if $2i\neq n$, the global section $f^*\eta^i:\OO_S\rightarrow R^i\pi_*\Omega^i_{\mathcal{X}/S}$ is an isomorphism.
\item the coherent sheaves $R^i\pi_*\Omega^j_{\mathcal{X}/S}$ are locally free for any $i, j$, and any base change map is isomorphic,
\item the relative Hodge-de Rham spectral sequence
\begin{equation}
E^{j,i}_1=R^i\pi_*\Omega^j_{\mathcal{X}/S}\Longrightarrow \mathbb{R}^{i+j}\pi_*(\Omega_{\mathcal{X}/S}^\bullet)
\end{equation}
degenerates at the level $E_1$, and any base change map of the spectral sequences are isomorphic.
\end{enumerate}
\end{theorem}

\begin{proof}
Recall Lemma \ref{lemma5.4} that
\[
R^i\pi_*\Omega^j_{\mathcal{X}/S}= R^ip_*(\Omega^j_{\PP^n_S/S}\oplus \bigoplus\limits_{\mu=1}^{k-1}\Omega^j_{\PP^n_S/S}(\log D)\otimes \LL^{-\mu}).
\]
It follows from the claim (\ref{eqeqzero}) in Proposition \ref{prop5.5} that 
\[
R^ip_*(\Omega^j_{\PP^n_S/S}(\log D)\otimes \LL^{-\mu})=0, \text{~for~} i+j\leq n-1.
\] 
Therefore, $f^*: R^ip_*\Omega^j_{\PP^n_S/S} \rightarrow R^i\pi_*\Omega^j_{\mathcal{X}/S}$ is an isomorphism for $i+j\leq n-1$. Hence the first two assertions holds for $i+j\leq n-1$. By Serre duality(\cite[III 12, p. 211]{Har66} ), there is a canonical isomorphism
\begin{align*}
\wp: R^i\pi_*\Omega^j_{\XX/S}\rightarrow \mathcal{H}om(R^{n-i}\pi_*\Omega^{n-j}_{\XX/S}, \OO_S)
\end{align*} 
Therefore, the first two assertions extend to the cases of $i+j\geq n+1$.


The third assertion is deduced from $(1)$ and $(2)$. In fact, the function of Euler characteristic of $s\mapsto \chi(\Omega^j_{X_s})$ is locally constant on $S$. Notice that $R^i\pi_*\Omega^j_{\mathcal{X}/S}$ are locally free for all $i+j\neq n$. Hence, for a fixed integer $j$, the upper semi-continuous function
\[
s\mapsto \dim H^{n-j}(\XX_s, \Omega^j_{\XX_s})
\]
is locally constant on the $S$. So the coherent sheaf $R^{i}\pi_*\Omega^j_{\XX/S}$ is locally free for $i=n-j$. In addition, the locally freeness implies the base change maps are isomorphisms.

To prove $(4)$, we mimic Deligne's argument to prove the degenerate of the relative Hodge-de Rham spectral sequence of complete intersections in projective spaces. To be precise, one can construct a scheme $\tilde{S}$ which is smooth and of finite type over $\Spec \mathbb{Z}$ and a smooth family of $k$-fold cyclic coverings $\tilde{f}: \tilde{\XX}\rightarrow \tilde{S}$ satisfying the cartesian diagram

\begin{equation}\label{basechange}
\xymatrix{
\mathcal{X}\ar[r] \ar[d]_\pi&\tilde{\XX} \ar[d]_{\tilde{\pi}}\\
S \ar[r]& \tilde{S}.
}
\end{equation}
The construction of the scheme $\tilde{S}$ can be found in \cite[6, p. 130]{Hodge}. Let $\eta$ be the generic point of $\tilde{S}$. The characteristic of the fraction field $K(\eta)$ is zero. Thus the induced Hodge-de Rham spectral sequence
\begin{equation}\label{degen}
E^{j,i}_{1,\eta}:=R^i\tilde{\pi}_*\Omega^j_{\tilde{X}_\eta}\Longrightarrow R^{i+j}\tilde{\pi}_*\Omega^{\bullet}_{\tilde{\XX}_\eta}
\end{equation} degenerates at the level $\tilde{E}_1$. Let
\[
d^{j,i}_1: R^i\tilde{\pi}_*\Omega^j_{\tilde{X}_S}\rightarrow  R^i\tilde{\pi}_*\Omega^{j+1}_{\tilde{X}_S}
\]
be the first differential map of the spectral sequence. Then $d^{j,i}_1$ vanishes on $\tilde{S}$ because it is zero on the generic point. Therefore, the relative Hodge-de Rham spectral sequence degenerates at $E_1$.

\end{proof}

\subsection{The Infinitesimal Torelli Theorem}

In section 3, we have the infinitesimal Torelli theorem of a cyclic covering $X$ of $\PP^n_{\mathbb{C}}$, see Theorem \ref{thm3.5}. In fact, this conclusion can be generalized to $X$ over an arbitrary field. The method to prove Theorem \ref{thm3.5} is to verify Flenner's criterion of the infinitesimal Torelli theorem \cite[Theorem 1.1]{Fle86}. The following theorem shows this criterion turns out to be algebraic.

\begin{theorem}\cite[Appendix A]{Pan15}\label{thm5.8}
Let X be a smooth proper scheme of dimension $n$ over a field $K$. Assume the existence of a resolution of $\Omega_{X/K}^1$ by locally free sheaves\[0\rightarrow \mathcal{G}\rightarrow \mathcal{F}\rightarrow\Omega_{X/K}^1\rightarrow 0.\] Let $D_r\mathcal{G}$ be the divided power $\Sym^r(\mathcal{G^{\vee}})^{\vee}$, and let $\kappa_X$ be the canonical sheaf of $X$. If following two conditions:
\begin{enumerate}
\item $H^{j+1}(X, \Sym^{j}\mathcal{G}\otimes \Lambda^{n-j-1}\mathcal{F}\otimes \kappa_X^{-1})=0$ for $0\leq j\leq n-2$;
\item the pairing \[H^0(X, D_{n-p}(\mathcal{G}^{\vee})\otimes \kappa_X)\otimes H^0(X, D_{p-1}(\mathcal{G}^{\vee})\otimes \kappa_X)\rightarrow H^0(X, D_{n-1}(\mathcal{G}^{\vee})\otimes \kappa_X^2)\] is surjective for a positive integer $p$ no larger than $n$
\end{enumerate}
are satisfied, then the cup product map
\begin{equation} \label{cupproduct}
\lambda_p: H^1(X, \Theta_X)\rightarrow \Hom(H^{n-p}(X, \Omega_{X/K}^p), H^{n+1-p}(X,\Omega_{X/K}^{p-1}))
\end{equation} is injective.
\end{theorem}


\begin{theorem}\label{thm5.9}
Let X be a smooth $k$-fold cyclic covering of $\PP^n_{K}$ branched along the smooth divisor $D$ over a field $K$. Suppose that $n$ is at least $2$ and $k$ is prime to $\Char(K)$. Then the cup product (\ref{cupproduct}) is injective for $X$ with the only exceptions
\begin{itemize}
\item $X$ is a 3-fold covering of $\PP^2_K$ branched along a cubic curve;
\item $X$ is a 2-fold covering of $\PP^2_K$ branched along a quartic curve;
\end{itemize} 
\end{theorem}

\begin{proof}

If $X$ is a hypersurface of $\PP^{n+1}_K$ but not a cubic surface, the infinitesimal Torelli theorem of $X$ had been proved in \cite[Proposition A.9.]{Pan15}. Therefore, we may assume that $X$ is not a hypersurface. Recall the notation in (\ref{linebundleL}), $X$ is a smooth divisor of the $\PP^1$-bundle $\hat{L}$ with normal bundle $N_{X/\hat{L}}=f^*\LL^{k}$. We thus apply Theorem \ref{thm5.8} to the following natural resolution
\begin{equation}\label{5.9.2}
0\rightarrow f^*\LL^{-k}\rightarrow \Omega^1_{\hat{L}}|_X\rightarrow \Omega^1_{X/K}\rightarrow 0.
\end{equation}
The computations follow line by line of the computations over the field of complex numbers, see \cite[Theorem 4.8]{Weh86}. There two cases we need to take care because the computations involve the characteristics of the field. In the following, we show that these two cases are not exclusive. We set $\LL=\OO(m)$, where $m>0$.

$\bullet$ $(n,m,k)=(3,2,2)$. By the calculation in \cite[Theorem 4.8]{Weh86}, the condition (1) in Theorem \ref{thm5.8} is equivalent to
\[
H^1(X, \Omega^{2}_{\hat{L}}|_X\otimes \kappa^{-1}_X)=0.
\]
To compute the cohomology group $H^1(X, \Omega^{2}_{\hat{L}}|_X\otimes \kappa^{-1}_X)$, it uses the natural short exact sequence
\[
0\rightarrow f^*\Omega^2_{\PP^3}\rightarrow \Omega^2_{\hat{L}}|_X\rightarrow f^*(\Omega^1_{\PP^3}\otimes \mathcal{L}^{-1})\rightarrow 0.
\] induced by (\ref{5.9.2}).
Note that $\kappa_X=f^*(\kappa_{\PP^3}\otimes \LL)$ by Proposition \ref{prop2.2} (ii), it follows that
\begin{align*}
H^1(X, f^*(\Omega^1_{\PP^3}\otimes \mathcal{L}^{-1})\otimes \kappa^{-1}_X)&=H^1(\PP^3, \Omega^1_{\PP^3}\otimes\kappa^{-1}_{\PP^3}\otimes\LL^{-2}\otimes f_*\OO_X)\\
&=\bigoplus\limits^{1}_{r=0}H^1(\PP^3, \Omega^1_{\PP^3}(-2r))=H^1(\PP^3, \Omega^1_{\PP^3}),\\
\text{~and~} 
H^2(X, f^*\Omega^2_{\PP^3}\otimes \kappa^{-1}_X)&=H^2(\PP^3, \Omega^2_{\PP^3}(2)\otimes f_*\OO_X)\\
&=\bigoplus\limits^{1}_{r=0}H^2(\PP^3, \Omega^2_{\PP^3}(2-2r))=H^2(\PP^3, \Omega^2_{\PP^3}).
\end{align*}
Therefore, we obtain the exact sequence
\[
0\rightarrow H^1(X, g^*\Omega^2_{\hat{L}}\otimes \kappa^{-1}_X)\rightarrow H^1(\PP^3, \Omega^1_{\PP^3})\xrightarrow{\delta} H^2(\PP^3, \Omega^2_{\PP^3}).
\] 
The connecting morphism $\delta$ is the cup product with the frist Chern class $c_1(\mathscr{L}_D)$, see \cite[p. 470]{Weh86}. Note that the degree of $D$ is 4. The connecting morphism is injective if $\Char(K)\neq 2$. Recall that the smoothness of $X$ requires that $k$ is prime to $\Char(K)$. Therefore, $\Char(K)\neq 2$ in our case and we obtain $
H^1(X, g^*\Omega^{2}_{\hat{L}}\otimes \kappa^{-1}_X)=0.$\\
 
$\bullet$ $(n, m, k)=(2, 2, 3)$. In this case, the canonical sheaf $\kappa_X=f^*\OO_{\PP^2}(1)$ is ample. We refer to a theorem \cite[Theorem $1'$]{LWP77} characterizing the cup product 
\[
\lambda_2: H^1(X, \Theta_X)\rightarrow \Hom(H^{0}(X, \Omega_X^2), H^{1}(X,\Omega_X^1))
\] is injective. We note that the proof of \cite[Theorem $1'$]{LWP77} is algebraic and holds for any characteristic though the statement is for a complex compact K\"ahler manifold. Moreover, the first two assumptions in \cite[ Theorem 1]{LWP77} imply the hypothesis in \cite[Theorem $1'$]{LWP77}. Then in our case, it suffices to verify the second assumption of \cite[ Theorem 1]{LWP77}, which is equivalent to verify $H^0(X, \Omega^1_{X}\otimes \OO_{X}(1))=0$. By the projection formula and Lemma \ref{lemma5.4}, we have
\begin{align*}
H^0(X, \Omega^1_X\otimes \OO_X(1))&=H^0(\PP^2, \OO_{\PP^2}(1)\otimes f_*\Omega^1_{X})\\
&=H^0(\PP^2, \Omega^1_{\PP^2}(1))\oplus \bigoplus\limits_{i=1}^{2}H^0(\PP^2, \Omega^1_{\PP^2}(\log D)\otimes \OO_{\PP^2}(1)\otimes \LL^{-i})\\
&=\bigoplus\limits_{i=1}^{2}H^0(\PP^2, \Omega^1_{\PP^2}(\log D)\otimes \OO_{\PP^2}(1)\otimes \LL^{-i}).
\end{align*} 
Notice that the invertible sheaf $\OO_{\PP^2}(1)\otimes \LL^{-i}$ is negative. It follows from Proposition \ref{prop5.5}, cf. (\ref{eqeqzero}) that
\[
H^0(\PP^2, \Omega^1_{\PP^2}(\log D)\otimes \OO_{\PP^2}(1)\otimes \LL^{-i})=0,~ 1\leq i\leq 2.
\]
Therefore we checked that the infinitesimal Torelli theorem holds for these two cases.
\end{proof}

\section{Automorphisms in Positive Characteristic}
Let $K$ be an algebraically closed field of positive characteristic. We recall the main theorem of the paper \cite{Pan16}.

\begin{theorem}\cite[Theorem 1.7]{Pan16}\label{corolifting}
Let $\bar{X}$ be a smooth projective scheme over the Witt ring $W:=W(K)$ of $K$, and let $X$ be the special fiber over $\Spec(K)$. Assume that the relative Hodge-de Rham spectral sequences of $\bar{X}/W$ degenerates at $E_1$, and all the terms in the spectral sequence are locally free. Assume that $g_0$ is an automorphism of $X$ such that the map 
\[
\cris^n(g_0): \cris^n(X/W)\rightarrow \cris^n(X/W)
\] 
preserves the Hodge filtrations induced by the natural identification 
\[
\cris^n(X/W)\cong \Ho^n_{\DR}(\bar{X}/W).
\]
If the cup product\[\Ho^1(X,T_{X})\rightarrow
\bigoplus_{p+q=n} \mathrm{Hom}(\Ho^q(X,\Omega^p_{X/K}),\Ho^{q+1}(X,\Omega^{p-1}_{X/K}))\]is injective, then one can lift $g_0$ to an automorphism $g: \bar{X}\rightarrow \bar{X}$ of $\bar{X}$ over $W$. 
\end{theorem}

\begin{lemma}\label{vanishing}
Let X be a smooth $k$-cyclic covering of $\PP^n_{K}$ branched along a smooth hypersurface $D$ and $n\geq 2$. Then $H^1(X,\OO_X)=H^0(X,\Omega_X^1)=0$. Moreover, if $n=2$ and the canonical bundle $\kappa_X$ is ample or trivial, then $H^0(X, T_X)=0$.
\end{lemma}

\begin{proof}
It follows from  Lemma \ref{lemma5.4} and (\ref{eqeqzero}) that $H^1(X, \OO_X)$ and $H^0(X, \Omega^1_X)$ are zero. Furthermore, we have
\[
H^0(X, T_X)=H^2(X,\Omega_{X/K}^1\otimes \kappa_X)
\]
Note that $X$ can be lift to the Witt ring $W(K)$ of $K$ as a smooth cyclic covering of $\PP^n_W$. By Theorem \ref{thm5.4}, the relative Hodge-de Rham spectral sequences degenerate in $E_1$. So it suffices to prove the vanishing of $H^2(X,\Omega_{X/K}^1\otimes \kappa_X)$ for $X$ over the generic fiber $X_{\CC}$. Notice that $\kappa_X$ is ample, the vanishing is a consequence of the Kodaira-Akizuki-Nakano vanishing theorem. If $\kappa_X$ is trivial and $n=2$, then $X$ is a $K3$ surface. It is well known that $H^0(X, T_X)=0$
\end{proof}

\begin{lemma} \label{torsionfree}
Let X be a smooth $k$-cyclic covering over $\PP^n_{K}$ branched along $D(\subseteq \PP^n_K)$ with $n\geq 2$. Then the N\'eron-Severi group $\NS(X)$ is torsion free. 

\end{lemma}
\begin{proof}

In fact, by the universal coefficient theorem of crystalline cohomology, we have a short exact sequence
\begin{equation}\label {uc}
0\rightarrow \cris^1({X}/W)\otimes_W K\rightarrow \Ho_{\DR}^1({X}/K)\rightarrow \mathrm{Tor}_1^W(\cris^{2}({X}/W),K) \rightarrow 0
\end{equation}
where $W$ is the Witt ring of $K$. It follows from Lemma \ref{lemma5.4} and (\ref{eqeqzero}) that $H^1(X, \OO_X)$ and $H^0(X, \Omega^1_X)$ are zero. By the Hodge-de Rham spectral equence,  $H^1_{DR}(X/K)=0$. It implies $\mathrm{Tor}_1^W(\cris^{2}({X}/W),K)=0$, in other words, the crystalline cohomology $\cris^{2}({X}/W)$ is p-torsion free. By a theorem of Illusie and Deligne, see \cite[Remark 3.5]{delsur} and \cite{Ill}, we have an injection \[\NS(X)\otimes \mathbb{Z}_p \hookrightarrow \cris^2(X/W).\] We conclude that $\NS(X)$ is $p$-torsion-free.

On the other hand, for any prime $l\neq p$, we have the short exact sequence \cite[Chapter V, Remark 3.29 (d)]{Milne}
\begin{equation} \label{NSinj}
0\rightarrow \NS(X)\otimes \Z_l\rightarrow \et^2(X,\Z_l(1))\rightarrow T_l(\mathrm{Br}(X))\rightarrow 0.
\end{equation} We claim that $\et^2(X,\Z_l(1))$ is torsion free. Therefore, the group $\NS(X)$ is torsion-free.

In fact, we denote the natural lifting of $X$ over $W$ by $\overline{X}$. Choose an embedding $W\rightarrow \mathbb{C}$. We have the variety $\X_{\mathbb{C}}$ which is a $k$-cyclic covering over $\PP^n_{\mathbb{C}}$. Since a cyclic covering over a projective space is a hypersurface in a weighted projective space, it is simply connected by \cite[Theorem 3.2.4 (ii)']{WPS}. 

By the universal coefficient theorem, we have
\begin{align*}
\Ho^2_{sing}(\X_{\mathbb{C}},\Z_l)&=\Hom(\Ho_2(\X_{\mathbb{C}},\Z),\Z_l)=\lim_{\overleftarrow n } \Hom(\Ho_2(\X_{\mathbb{C}},\Z), \Z/l^n\Z)\\
&=\lim_{\overleftarrow{n}} \et^2 (\X_{\mathbb{C}},\Z/l^n\Z)=\et^2(\X_{\mathbb{C}},\Z_l)=\et^2(X,\Z_l).
\end{align*}
Since $\Ho^2_{sing}(\X_{\mathbb{C}},\Z_l)$ is torsion free, the group $\et^2(X,\Z_l)$ is torsion-free. The claim holds for $X$ over $K$.

\end{proof}

Denote by $\Aut(X)_{tr}$ the kernel of $\Aut(X)\rightarrow \et^n(X,\Ql)$ ($l\neq char(K)$).

\begin{lemma} \label{autofinite}
Let X be a smooth $k$-cyclic covering over $\PP^n_{K}$ branched along $D$ with $n\geq 2$. If $X$ is not a quadric hypersurface, i.e. $\deg D\geq 3$, then $\Aut(X)_{tr}$ is finite.

\end{lemma}
\begin{proof}
By Lemma \ref{vanishing} and Lemma \ref{torsionfree}, we conclude that $\NS(X)=\Pic(X)$ is torsion free. Therefore, the following map
\begin{equation}\label{cherninjective}
c_1:\Pic(X)\rightarrow \Pic(X)\otimes \Z_l \rightarrow  \et^2(X,\Z_l(1))
\end{equation}
is injective by (\ref{NSinj}) where $l$ is a prime different from $\mathrm{char} (K)$.
\begin{itemize}
\item For $n\geq 3$. We claim $\Pic(X)=\Z$. In fact, we can lift $X$ to a cyclic covering $\X$ over complex numbers $\mathbb{C}$ with $\et^2(X,\Ql)=\et^2(\X,\Ql)$. By Proposition \ref{prop4.6}, we have $\et^2(X,\Ql)=\Ql$. The claim follows from Lemma \ref{torsionfree}. (If the degree of $D$ is $3$, then $X$ is a cubic hypersurface of dimension at least $3$ and the statement still holds by the Grothendieck-Lefschetz theorem). The claim implies that every automorphism preserves the ample line bundle $f^*\OO_{\PP^n}(1)$. Note that $\Aut_L(D)$ is finite if $\deg(D)\geq 3$ (\cite[Theorem 1.3]{Poo05}). It follows from Lemma \ref{lemma4.2}, Lemma \ref{lemma4.3}, and Proposition \ref{prop4.3} that $\Aut(X)$ is finite.

\item For $n=2$. We choose a very ample line bundle $L$ on $X$ such that the complete linear system $|L|:X\rightarrow \PP^N$ induces an embedding. It follows from the injectivity of the map $c_1$ (\ref{cherninjective}) and the torsion-freeness of $\Pic(X)$ that every automorphism $f\in \Aut_{tr}(X)$ fixes the line bundle $L$, i.e., $f^*L=L$. Therefore, we have a linear automorphism $g$ inducing the following diagram
\[
\xymatrix{ 
X\ar[r]^{|L|} \ar[d]_f & \PP^N \ar[d]^g  \\
X\ar[r]^{|L|}  & \PP^N .
}
\]
Let $G=\{h\in \mathrm{PGL_{N+1}}|h(X)=X\}$. We have $\Aut_{tr}(X)\subseteq G$. We claim $G$ is an subalgebraic group of $\mathrm{PGL_{N+1}}$. In fact, consider the Hilbert scheme $\Hilb_{p_X(t)}(\PP^N)$ parametrizes $X$. There is a natural action of $\mathrm{PGL_{N+1}}$ on $\Hilb_{p_X(t)}(\PP^N)$. The stabilizer of this action of the point $[X]$ parametrizing $X$ is $G$. Therefore, $G$ is algebraic. On the other hand, the infinitesimal automorphism of $X$ is trivial, i.e., $H^0(X,T_X)=0$ (cf. Lemma \ref{vanishing}). It follows that $G$ is a subgroup scheme of $\mathrm{PGL_{N+1}}$ with $\dim(G)=0$, hence, it is finite. We conclude that $\Aut_{tr}(X)$ is finite.
\end{itemize}

\end{proof}

\begin{theorem} \label{autopcfaith}
Let $K$ be an algebraically closed field of positive characteristic, and let X be a smooth cyclic covering of $\PP^n_{K}$ with $n\geq 2$. If $X$ is not a quadric hypersurface, then the action of the automorphism group $\Aut(X)$ on $\et^n(X,\Ql)$ ($l\neq char(K)$) is faithful, i.e., the natural map \[\Aut(X)\rightarrow \Aut(\et^n(X,\Ql))\] is injective.

\end{theorem}

\begin{proof}
Suppose that $g_0$ is an automorphism of $X$ and in $\Aut(X)_{tr}$.
It follows from  \ref{autofinite} that $g_0$ is of finite order. Let $W$ be the Witt ring of $K$. Note that \[\det(\Id-g_0^*t,\cris^n(X/W)_{W[\frac{1}{p}]})=\det(\Id-g_0^*t,\et^n(X,\Ql)),\]
see \cite[Theorem 2]{KM} and \cite[3.7.3 and 3.10]{Ill1}. The finiteness of the order of $g_0$ implies that $\et^n(g_0,{\mathbb{Q}_l})=\Id$ if and only if $\cris^n(g_0)_{W[\frac{1}{p}]}=\Id$ since both $\et^n(g_0,{\mathbb{Q}_l})$ and $\cris^n(g_0)_{W[\frac{1}{p}]}$ are diagonalizable. Note that $X$ can be lift to $W$ as a smooth cyclic covering of $\PP^n_W$. Let $\X$ be a such lifting of the $X$ over $W$. It follows from Theorem \ref{thm5.4} that $$\cris^n(X/W)=\Ho^n_{\DR}(\overline{X}/W)$$ is a finite free $W$-module. Therefore, we have $\cris^n(g_0)=\Id$. 

By Theorem \ref{thm5.4} and Theorem \ref{thm5.9}, the assumptions of Proposition \ref{corolifting} hold for $g_0$ except when $X$ is
\begin{itemize}
\item a 3-fold covering of $\PP^2_K$ branched along a cubic curve,
\item a 2-fold covering of $\PP^2_K$ branched along a quartic curve.
\end{itemize} 
For all other cases, one can lift the automorphism $g_0$ to an automorphism $g$ of $\X/W$. Then the proof in Theorem \ref{autofaith} can be applied to obtain the conclusions.

Recall that the canonical bundle is
\[
\kappa_X=f^*(\kappa_{\PP^2}\otimes \LL^{k-1}).
\]
The canonical bundles of the above two excluded cases are anti-ample, i.e. where $X$ is a Fano surface. The possible types of $X$ are listed in the proof of Proposition \ref{propfaithsurface}. Since every Fano surface is a blowup of the projective plane, we can use the same argument in Proposition \ref{propfaithsurface} to prove the theorem for these two cases .

\end{proof}

\end{document}